\documentclass[a4paper]{amsart}
\usepackage[T1]{fontenc}
\usepackage[utf8]{inputenc}

\usepackage[numbers]{natbib}
\usepackage{doi}
\usepackage{graphicx}
\usepackage[ngerman,american]{babel}
\usepackage{amssymb,amsmath,amsthm}
\usepackage{mathtools}
\usepackage{enumerate}
\usepackage{siunitx}
\usepackage{hyperref}
\usepackage{cleveref}
\usepackage[font=small]{subcaption}
\usepackage{float}
\newfloat{algorithm}{t}{lop}

\usepackage{algorithm}
\usepackage{algorithmic}

\usepackage{cleveref}
\usepackage{todonotes}

\newtheoremstyle{mythmstyle}
{1em} 
{1em} 
{\itshape} 
{} 
{\bfseries} 
{.} 
{ } 
{} 
\newtheoremstyle{myremarkstyle}
{1em} 
{1em} 
{\rmfamily} 
{} 
{\bfseries} 
{.} 
{ } 
{} 
\theoremstyle{mythmstyle}
\newtheorem{lemma}{Lemma}

\theoremstyle{myremarkstyle}
\newtheorem{remark}{Remark}

\newcommand{\curl}{\nabla \times }
\newcommand{\hcurl}{\widehat{\nabla} \times }
\newcommand{\bu}{\mathbf{u}}
\newcommand{\bv}{\mathbf{v}}
\newcommand{\bz}{\mathbf{z}}
\newcommand{\bU}{\mathbf{U}}

\begin{document}
	
	%
	\title[Numerical Shape Eigenvalue Optimization]{Numerical Eigenvalue
		Optimization by Shape-Variations for Maxwell's Eigenvalue Problem}
	
	\author{Christine Herter}
	\address{Universität Hamburg, Fachbereich Mathematik, Bundesstr. 55,
		20146 Hamburg}
	\email{christine.herter@uni-hamburg.de}
	\author{Sebastian Schöps}
	\address{Schloßgartenstr. 8
		64289 Darmstadt}
	\email{sebastian.schoeps@tu-darmstadt.de}
	\author{Winnifried  Wollner}
	\address{Universität Hamburg, Fachbereich Mathematik, Bundesstr. 55,
		20146 Hamburg}
	\email{winnifried.wollner@uni-hamburg.de}
	
	


\begin{abstract}
	In this paper we consider the free-form optimization of eigenvalues in electromagnetic systems by means of shape-variations with respect to small deformations. 
	The objective is to optimize a particular eigenvalue to a target value.
	We introduce the mixed variational formulation of the Maxwell eigenvalue problem introduced
	by Kikuchi (1987) in function spaces of \(H(\operatorname{curl}; \Omega)\) and \(H^1(\Omega)\). 
	To handle this formulation, 
	suitable transformations of these spaces are utilized, e.g., of Piola-type for the space of \(H(\operatorname{curl}; \Omega)\). 
	This allows for a formulation of the
	problem on a fixed reference domain together with a domain mapping.
	Local uniqueness of the solution is obtained by a normalization of the the eigenfunctions. This allows us to derive adjoint formulas for the derivatives of the eigenvalues with respect to domain variations. For the solution of the resulting optimization problem, we develop a particular
	damped inverse BFGS method that allows for an easy line search procedure while retaining
	positive definiteness of the inverse Hessian approximation.
	The infinite dimensional problem is discretized by mixed finite elements
	and a numerical example shows the efficiency of the proposed approach.
\end{abstract}


	\maketitle
	
\section{Introduction}
In this paper, we consider a numerical eigenvalue optimization by shape-variations for Maxwell's eigenvalue problem. This work is motivated by particle accelerator cavities, where an accurate geometry is essential for the acceleration of the particles.	
Eigenvalue optimization of resonating structures, e.g., in context of electromagnetic cavities, is a challenging area of research because the therein resulting eigenmodes are easily effected to small deformations of the domain. 
The goal of electromagnetic cavities is to achieve an acceleration eigenmode for the particles.
The most relevant eigenmode is the fundamental Transverse Magnetic (TM) mode, shown in Figure~\ref{tmMode}.  
For a detailed description of such a cavity as well as its associated components, we refer to the paper about Superconducting TESLA cavities by~\cite{PhysRevSTAB}.
\begin{figure}[ht!]
	\centering
	\includegraphics[width = 0.65\linewidth]{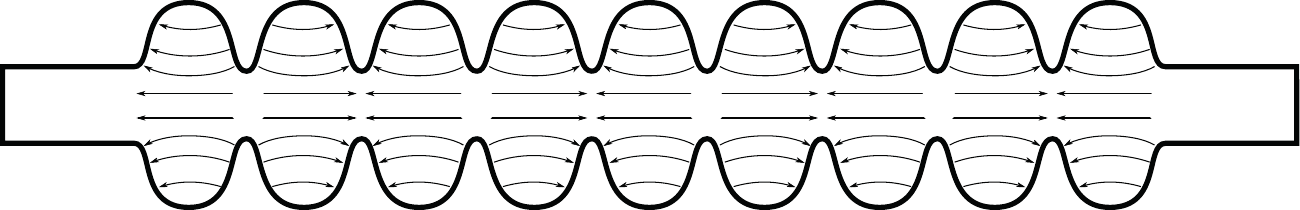}
	\caption{Electric field of the $\mathrm{TM}_{010}$, also called $\pi$-mode}
	\label{tmMode}
\end{figure}

The analytic sensitives of eigenpairs of Maxwell's eigenvalue problem have been considered in isogeometric analysis, e.g., in work of~\cite{Anna_Ziegler_2023}, where the sensitivities of the eigenpair have been computed with respect to a scalar parameter. 
The therein parameter optimization has been done with a simplified formulation of Maxwell's eigenvalue problem excluding the divergence-free constraint on the eigenfunction.
Due to the formulation, care has to be taken in removing spurious solutions.
Further, the deformation and its derivatives are expressed in terms of parameters explicitly dependent on weights and control points.
Further work of optimizing resonant structures on this approach has been seen in, e.g.,~\cite{Lewis2003TheMO}.
Goals for the optimization of these cavities are uncertainty quantification, see, e.g.,~\cite{UncertaintyQuantiNiklas}, and mode tracking, see e.g.,~\cite{modetracking_jorkowski_2018}.
In context of cavities, there already exists several references to multi-objective optimization, e.g., \cite{multiobj2019rienen,optimizationCavitiy2023rienen} as well as parameter optimization, e.g., \cite{PhysRevSTAB.12.114701}. Further, for solving multi-objective optimization problems, we refer to \cite{MultiobjectiveShapeOptimizationEvolutionaryAlgo_2019_kranj} for the use of an evolutionary algorithm, and to \cite{PUTEK2024113125}, addressing a stochastic Maxwell's eigenvalue problem.

In this paper, we consider a free-form eigenvalue optimization problem with respect to a mixed
variational formulation of Maxwell's eigenvalue problem by~\cite{Kikuchi1987MixedAP}, i.e., for a certain eigenvalue $\lambda$ of the Maxwell operator
\begin{align*}
	\nabla \times ( \nabla \times u) &= \lambda u && \text{in } \Omega, \\
	\nabla \cdot u &= 0 && \text{in } \Omega, \\
	\nu \times u &= 0 &&\text{on } \partial \Omega,\\
	\Vert u \Vert^2 &= 1, 
\end{align*}
which governs the field distribution of the eigenmodes and their associated frequency.
In this setup the eigenvalue, and corresponding eigenfunctions $u$, are to be optimized by
choice of the physical domain $\Omega$. 
The results and formulation have been announced in~\cite{cherter_PAMM}.
As detailed in, e.g.,~\cite{Boffi2010FiniteEA}, this mixed formulation prohibits spurious
eigenvalues. To allow for a larger variety in admissible domains,  
we consider the method of mappings allowing arbitrary, sufficiently regular and injective,
deformations of a reference domain $\widehat{\Omega} \subset \mathbb{R}^d$, $d\in \{2,3\}$.
For the method of mappings is based on kinematic statements from continuum mechanics, where
the actual physical domain $\Omega_q$ is then given by a \mbox{deformation
$q\colon \widehat{\Omega} \rightarrow \mathbb{R}^d$} by defining $\Omega_q = F(\widehat{\Omega}) = (I+q)\widehat{\Omega}$. 
For details, we refer to, e.g., \cite{murat_simon_1976_etude_problems_optimal_design, Murat1976SurLC}.

The method of mappings is already applied to several shape optimization problems to derive Fréchet derivatives with respect to shape-variations. For example, we refer to  \cite{Fischer2017Frechet}, where properties of Fréchet differentiability are derived in the context of Navier-Stokes problems.  
Due to the, infinite, dimension of the control space standard sensitivity based approaches for
derivatives of the eigenvalues, e.g.,~\cite{Lancaster:1964,AndrewChuLancaster:1993} are no longer
applicable. Hence, as announced in~\cite{cherter_PAMM}, we utilize the adjoint approach to
calculate eigenvalue derivatives.  
The origin of the adjoint approach is explained in~\cite{Murat1997}, and perspective of Lagrange multiplier is introduced in~\cite{cea_1986_calcul}.
For the adjoint approach in context of optimization problems governed by partial differential equations, we refer to, e.g.,~\cite{Troeltzsch:2010,optPDE2009UlbrichHinze}.
Moreover, the adjoint method has already been applied to many applications of control problems as well as shape and topology optimization problems, e.g., in the domain of structural optimization of solid bodies, see \cite{ALLAIRE_2004_structureoptimization, ALLAIRE_2008_stressoptdesign, ALLAIRE_2011_damage, ALLAIRE_2014_shapeopt}, or in fluid mechanics, see, e.g., \cite{Heners2018259}. Further, optimization problems where the functionals depend on the eigenvalues are studied in, e.g., \cite{Kikuchi1995} and discussed in \cite{Opt_ev_adjoint_ToaderBarbarosie+2017+142+158}.
The particular formulation of the eigenvalue problem and its adjoint are inspired
by~\cite{heuveline_rannacher_a_posteriori_error_control_evp,RannacherWestenbergerWollner+2010_adaptive_fem_evp},
where this approach is applied to a posteriori error estimation for elliptic eigenvalue problems.
For more literature for the computation of shape derivatives in context of PDE constrained optimization, we refer to \cite{delfour2011shapes,Sokolowski1992shape,haslinge2003shapeintro}.

In order the concerning optimization problem, we consider a damped inverse BFGS method for infinite dimensional problems. 
The damping is based on an
idea of~\cite{Powell:1978} for the standard BFGS method, see
also~\cite{Numerical_opt_NoceWrig06}. Our implementation of the damped inverse BFGS method is
already used in~\cite{BezbaruahMaierWollner:2023}.

This paper is structured as follows. In \Cref{sec:optproblem}, we model the optimization problem of Maxwell's eigenvalue problem. Therefore, we consider the method of mappings in \Cref{subsec::mapping}. Based on this method, we derive the shape optimization problem depending on domain variations in \Cref{subsec:shapeopt}.
Afterwards, we discuss the adjoint calculus for a generalized eigenvalue optimization problem in \Cref{sec:adjointCalc_generalProblem}, where we consider an adjoint representation for the derivatives in \Cref{subsec:reducedderivative} and derive the formulas for Maxwell's eigenvalue optimization problem with domain transformations in \Cref{subsec:adjointCalc_MaxwellProblem}.	
In Section~\ref{sec:DampedInverseBFGS}, we discuss the optimization method used for the
solution of the Maxwell eigenvalue optimization problem. In particular, we will introduce a
damped inverse BFGS method. Here, the damping is utilized to allow the use of an
Armijo backtracking linesearch within the optimization while still guaranteeing
positive definiteness of the approximate inverse Hessian. 
Finally, in Section~\ref{sec:examples}, we provide the details of the finite element discretization using a mixed finite element method utilizing Lagrange and Nédélec elements to obtain
conforming subspaces of \(H_0^1(\widehat{\Omega})\) and \(H_0(\operatorname{curl}; \widehat{\Omega})\), see, e.g., Monk \cite{monk2003finite}. Further, several numerical examples show the correctness and
efficiency of our implementation.

\section{Modeling the Optimization Problem of Maxwell's Eigenvalue Problem}\label{sec:optproblem}
\subsection{Mapping}
\label{subsec::mapping}
Let \(d = 2,3 \), we denote by \(\widehat{\Omega} \subset \mathbb{R}^d \)
the, simply connected, reference domain. To map this domain to the physical domain
\(\Omega_q \subset \mathbb{R}^d\) on which the Maxwell eigenvalue problem is to be posed, 
we define a deformation depending on a displacement $q$ which we consider as our
control variable. 
The transformation of a point \(\widehat{x}\subset \widehat{\Omega} \) to a corresponding point \(x \subset \Omega_q\) is defined by 
\begin{equation*}
	x = \mathrm{F}_{q}(\widehat{x}) = q(\widehat{x}) + \widehat{x}.
\end{equation*}
To ensure invertibility of this mapping, as well as invertibility of induced maps between
function spaces on $\widehat{\Omega}$ and $\Omega_q$, we require that
\(q \in Q^{\rm{ad}} \subset Q := H^1(\widehat{\Omega};\mathbb{R}^d)\).
More precisely, defining the deformation gradient
\begin{equation}
\mathrm{DF}_q = \mathrm{I} + \nabla q, \label{deformationgradient}
\end{equation}
where $\mathrm{I}$ denotes the $d\times d$ identity matrix, and $\mathrm{J}_q = \det(\mathrm{DF}_q)$ the corresponding
determinant, we require that for \mbox{each $q \in Q^{\rm{ad}}$} it holds
\[
\mathrm{DF}_q \in L^\infty(\widehat{\Omega};\mathbb{R}^{d\times d}), \qquad \mathrm{J}_q > 0.
\]
Based on these mappings, we need to define suitable transformations to assert the needed function
spaces on $\widehat{\Omega}$ and $\Omega$ are isomorphic.
Here, we need to consider the function spaces
\begin{align*}
	H_0(\operatorname{curl}; \Omega) &= \left\{v\in L^2(\Omega)^d : \curl  v \in L^2 (\Omega)^n; \;v\times \nu = 0 \text{ on } \partial \Omega \right\},\\
	H^1_0 (\Omega) &= \{v \in H^1(\Omega) : v = 0 \text{ on }\partial \Omega \},
\end{align*}
where \(n = 1\) for \(d=2\) and \(n =3 \) for  \(d = 3\) and $\nu$ denotes the outward unit normal on $\partial \Omega$.

To this end, see, e.g.,~\cite{Boffi2010FiniteEA},
we transform a scalar function \(\widehat{p}\in H^1_0(\widehat{\Omega})\) to a scalar function \(p = \mathcal{H}_q(\widehat{p}) \in H^1_0(\Omega_q)\) by
\begin{equation*}
	p \circ  \mathrm{F}_{q}  = \widehat{p} \qquad\text{on }\widehat{\Omega}.
\end{equation*} 
Over and above that,  $H^1_0(\widehat{\Omega})$ and $H^1_0(\Omega_q)$ are isomorphic and for any such pair \(\widehat{p} \in H^1_0(\widehat{\Omega})\) and \({p}\in H^1_0({\Omega}_q)\), it holds 
\begin{equation*}
	\nabla p  \circ \mathrm{F}_q = \mathrm{DF}_{q}^{-T} \widehat{\nabla} \widehat{p} \qquad\text{on }\widehat{\Omega}
\end{equation*}
where $\widehat{\nabla}$ and $\nabla$ denote the gradient with respect to the reference
coordinates $\widehat{x}$ and the physical coordinates $x$.

For \( H_0(\widehat{\operatorname{curl}}; \widehat{\Omega)}\) we need to apply the covariant Piola mapping. 
Therefore, for \(\widehat{u}\in H_0(\widehat{\operatorname{curl}}; \widehat{\Omega})\)
we associate a function \(u=\mathcal{G}_q(\widehat{u}) \in H_0(\operatorname{curl};\Omega_q)\) by the mapping 
\begin{equation*}
	u \circ \mathrm{F}_{q} = \mathrm{DF}_{q}^{-T} \widehat{u}\qquad\text{on }\widehat{\Omega}.
\end{equation*}
Then for \(u \in H_0(\operatorname{curl};\Omega_q)\) and corresponding \(\widehat{u}\in H_0(\widehat{\operatorname{curl}}; \widehat{\Omega})\), it holds that in the three-dimensional case
\begin{equation*}
	\curl  u\circ \mathrm{F}_q = \frac{1}{\mathrm{J}_q} \mathrm{DF}_q \hcurl \widehat{u} \qquad\text{on }\widehat{\Omega}.
\end{equation*}
In the two-dimensional case it holds that
\begin{equation*}
	\curl  u \circ \mathrm{F}_q = \frac{1}{\mathrm{J}_q} \hcurl\widehat{u} \qquad\text{on }\widehat{\Omega}.
\end{equation*}
In addition, the local volume change between the two domains satisfies
\begin{equation*}
	\rm{d}x = \mathrm{J}_q\,\rm{d}\widehat{x}.
\end{equation*}
For further details we refer to, e.g.,~\cite{monk2003finite,HigherOrderNumMeth}. 

\subsection{Shape Optimization Problem}
\label{subsec:shapeopt}
Based on the domain mapping introduced in Section~\ref{subsec::mapping}.
We define the following shape optimization problem of Maxwell's eigenvalue problem.  
We look for the solution of the problem 
\begin{equation}\left.
	\begin{aligned}
		\min	\; &J(q,\lambda(q))\\
		\text{s.t. }
		\nabla \times ( \nabla \times u) &= \lambda u && \text{in } \Omega_q \\
		\nabla \cdot u &= 0 && \text{in } \Omega_q \\
		\nu \times u &= 0 &&\text{on } \partial \Omega_q\\
		\Vert u \Vert_{\Omega_q}^2 &= 1,
	\end{aligned}
	\right.
	\label{EVP}
\end{equation}
where a particular eigenvalue, e.g., smallest, is selected from the possible solutions of~\eqref{EVP}.
Here, and following, \(\Vert \cdot \Vert \) and $(\cdot,\cdot)$ are the usual \(L^2\) norm and scalar product on $\widehat{\Omega}$ while an index $\Omega_q$ refers to the respective quantity 
on $\Omega_q$. Furthermore, \(\nu\) is the outer unit normal vector to the boundary \(\partial \Omega_q\).
The generalized eigenvalue problem of~\eqref{EVP} is reformulated in weak form using the 
variational formulation by Kikuchi \cite{Kikuchi1987MixedAP}, also seen in \cite{Boffi2010FiniteEA}. It is to find an\mbox{ eigenvalue \(\lambda \in \mathbb{R} \)} and corresponding eigenvector \(0 \neq u \in H_0(\operatorname{curl}; \Omega_q) \) such that, for some \(\psi \in H_0^1(\Omega_q) \), 
\begin{equation}
	\left.\begin{aligned}
		(\curl u, \curl v )_{\Omega_q} + (\nabla \psi, v)_{\Omega_q}  &= \lambda  (u,v)_{\Omega_q}  &&\forall \; v \in H_0(\operatorname{curl}; \Omega_q) \\
		(u, \nabla \varphi)_{\Omega_q}  &= 0 &&\forall \; \varphi \in H_0^1(\Omega_q)
	\end{aligned}
	\right.
	\label{mixedFormulationKikuchi}
\end{equation}

Considering now~\eqref{mixedFormulationKikuchi} on the reference domain \(\widehat{\Omega}\)
and applying the domain mapping introduced in Section~\ref{subsec::mapping}, we derive the
following transformed equations on the reference domain.
Let $u,v \in H_0(\operatorname{curl};\Omega_q)$, $\varphi \in H^1_0(\Omega_q)$ and corresponding $\widehat{u},\widehat{v} \in H_0(\widehat{curl};\widehat{\Omega})$, $\widehat{\varphi} \in H^1_0(\widehat{\Omega})$ be arbitrary.
In case of \(d=3\) the curl-equation is given by
\begin{align*}
	a(q;\widehat{u},\widehat{v}) &\overset{d = 3}{\coloneqq} \int_{\Omega_q}	\left(\curl u, \curl v\right)\rm{d} x &&=  \int_{\widehat{\Omega}}\frac{1}{\rm{J}_q} \,\left(\mathrm{DF}_{q}\,\curl \widehat{u},\mathrm{DF}_{q}\,\curl \widehat{v}\right)\rm{d} \widehat{x}\\
	\intertext{while for \(d=2\), it is}
	a(q;\widehat{u},\widehat{v}) &\overset{ d = 2}{\coloneqq} \int_{\Omega_q}\left(\curl u, \curl v\right)\rm{d} x &&= \int_{\widehat{\Omega}} \frac{1}{\mathrm{J}_q} \left( \curl \widehat{u},\curl \widehat{v}\right)\rm{d} \widehat{x}.
	\intertext{Independent of \(d=2,3\), the remaining terms are transformed as follows.}
	b(q;\widehat{u},\widehat{\varphi}) &\coloneqq 
	\int_{\Omega_q}\left(u, \nabla \varphi\right) \rm{d} x &&= \int_{\widehat{\Omega}}\mathrm{J}_q \left(\mathrm{DF}_{q}^{-T}\,\widehat{u}, \mathrm{DF}_{q}^{-T}\,\nabla \widehat{\varphi}\right)\rm{d} \widehat{x},\\ 
	m(q;\widehat{u},\widehat{v}) &\coloneqq\int_{\Omega_q}\left(u,v\right) \rm{d} x &&= \int_{\widehat{\Omega}} \mathrm{J}_q\left(  \mathrm{DF}_{q}^{-T}\,\widehat{u},\mathrm{DF}_{q}^{-T}\,\widehat{v}\right)\rm{d} \widehat{x}.
\end{align*}

Hence the weak formulation of the eigenvalue problem becomes find
\[(\lambda, u, \psi) \in \mathbb{R}\times H_0(\widehat{\operatorname{curl}};\widehat{\Omega}) \times H^1_0(\widehat{\Omega})\]
such that 
\begin{equation*}
	\begin{aligned}
		a(q;u,v) + b(q;v,\psi) &= \lambda m(q;u,v) \qquad &&\forall v \in H_0(\widehat{\operatorname{curl}};\widehat{\Omega}),\\
		b(q;\varphi,u) &= 0 \qquad &&\forall \varphi \in H^1_0(\widehat{\Omega}).
	\end{aligned}
\end{equation*}
We introduce the form \(k\colon Q^{\rm{ad}}\times (H_0(\widehat{\operatorname{curl}};\widehat{\Omega}) \times H^1_0(\widehat{\Omega}))^2 \rightarrow \mathbb{R}\) for notational convenience, by
\[
k(q;(u,\psi),(v,\varphi)) := a(q;u,v) + b(v,\psi) + b(u,\varphi).
\]
The variational form~\eqref{mixedFormulationKikuchi} can be written as find
$(\lambda, u, \psi) \in \mathbb{R}\times H_0(\widehat{\operatorname{curl}};\widehat{\Omega}) \times H^1_0(\widehat{\Omega})$ solving
\begin{equation}\label{eq:generalized_evp}
	k(q;(u,\psi),(v,\varphi)) = \lambda\, m(q;u,v)\quad \forall (v,\varphi)\in H_0(\widehat{\operatorname{curl}};\widehat{\Omega}) \times H^1_0(\widehat{\Omega}).  
\end{equation}
We obtain the discretized formulation of \eqref{eq:generalized_evp} straight forward by utilization of the Galerkin method. 
\section{Adjoint Calculus for a Generalized Eigenvalue Optimization Problem}
\label{sec:adjointCalc_generalProblem}
\subsection{An Adjoint Representation for the Derivatives}
\label{subsec:reducedderivative}
We consider a general eigenvalue optimization problem
\begin{equation}
	\begin{aligned}
		&\min_{q,\lambda,u} \qquad \qquad J(q,\lambda(q))\\
		&\text{s.t.}\left\{
		\begin{aligned}
			k(q;\bu,\bv)&= \lambda(q) \, m(q;\bu,\bv) \qquad \forall \; \bv \in \bU\\
			m(q;\bu,\bu) &= 1,\\
			\bu &\in \bU,\\
			\lambda &\in \mathbb{R},\\
			q &\in Q^{\rm{ad}}.
		\end{aligned}
		\right.
	\end{aligned}
	\label{GeneralEigenvalueProblem}
\end{equation}
Here \(Q^{\rm{ad}} \) is the set of admissible controls and \(\bU\) is the space, where we search for eigenfunctions.
In the example considered in Section~\ref{subsec:shapeopt}, it is $\bU = H_0(\widehat{\text{curl};0},\widehat{\Omega})\times H^1_0(\widehat{\Omega})$ and functions can be decomposed as $\bu = (u,\psi)$, $\bv=(v,\varphi)$.
Further, we assume the cost functional \(J:Q \times \mathbb{R} \to \mathbb{R}\) to be continuously differentiable on $Q^{\rm{ad}} \times \mathbb{R}$. The forms 
\(k,m:Q\times \bU \times \bU \to \mathbb{R}\) are linear in the last two arguments. Moreover, $m$ is assumed to be symmetric with respect to the last two arguments. Further,
they are both differentiable on $Q^{\rm{ad}}\times \bU\times \bU$. Here, differentiability on
$Q^{\rm{ad}}$ is to be understood in terms of sufficiently regular perturbations
$p \in \widetilde{Q}$ such that smallness of $p$ implies that the \mbox{forms $k,m$}
remain well-defined.

For each \(q \in Q^{\rm{ad}}\), we assume that the eigenvalue problem
\(k(q,u)= \lambda(q) \, m(q,u)\) admits, a sequence \mbox{ $\lambda_0(q)\le \lambda_1(q) \le \ldots$}
of real eigenvalues, which holds true for the considered Maxwell eigenvalue problem, see, e.g., \cite[Chapter 4.7]{monk2003finite}.
To apply adjoint calculus, we assume that, after normalization, the selected eigenvalue and
eigenfunction pair $(\lambda,\bu)$ to be locally unique; and thus in particular the eigenvalue
to be simple. 
Moreover, we assume a gap of real eigenvalues between the chosen eigenvalue \(\lambda_i(q)\) and the others, i.e., without crossing, which implies that there \mbox{exists \(\rho > 0\),} independent of \(q\), such that
\begin{equation*}
	\vert \lambda_i(q) - \lambda_j(q) \vert \geq \rho \qquad \text{ for all } j \neq i.
\end{equation*}
The selection rule, e.g., $\lambda(q) = \lambda_0(q)$, is understood implicitly defining
a solution operator $S \colon Q^{\rm{ad}} \rightarrow \mathbb{R}\times \bU$ by defining 
$q \mapsto S(q) = (\lambda(q),\bu(q)) = (\lambda,\bu)$
to be the particularly chosen eigenvalue-eigenfunction pair. 

With the help of the solution operator, we can define the reduced problem
\begin{equation}
	\min_{q\in Q^{\rm{ad}}} j(q) := J(q,\lambda(q)).
	\label{reducedEigenvalueProblem}
\end{equation}
We can now utilize standard Lagrange techniques, see, e.g.,~\cite{Troeltzsch:2010,BeckerMeidnerVexler:2007},
to obtain a representation of the \mbox{derivative $j'(q)\in Q^*$} and the $Q$-gradient
$\nabla_Q j(q) \in Q$.
To this end, we introduce the Lagrangian
\begin{equation*}
	\mathcal{L} : Q^{\rm{ad}} \times (\mathbb{R} \times \bU) \times (\mathbb{R} \times \bU)
\end{equation*}
of problem \eqref{reducedEigenvalueProblem} as
\begin{align}
	\label{lagrangian_general_evp}
	\mathcal{L}(q,(\lambda, \bu), (\mu, \bz)) = J(q,\lambda) - k(q;\bu,\bz)+ \lambda \, m(q;\bu,\bz) +
	\mu \left( m(q;\bu,\bu)-1\right)
\end{align}
where \(\bz \in \bU\) will be the adjoint state and $\mu$ is the multiplier for the
normalization of the eigenfunction condition.

Then as usual, we can recover the state equation (eigenvalue/eigenfunction) from
the condition
\begin{align*}
	0 &= \mathcal{L}_{(\mu,\bz)}'(q,(\lambda, \bu), (\mu, \bz))(\tau,\bv) & \forall (\tau,\bv) &\in \mathbb{R}\times \bU\\
	&= -k(q;\bu,\bv) + \lambda\,m(q;\bu,\bv)\\
	&\quad+ \tau (m(q;\bu,\bu)-1)\\
	\intertext{	using linearity of $k,m$ in the last argument.
		The adjoint state is then defined by}
	0 &= \mathcal{L}_{(\lambda,\bu)}'(q,(\lambda, \bu), (\mu, \bz))(\tau,\bv) & \forall (\tau,\bv) &\in \mathbb{R}\times \bU\\
	&= J_{\lambda}'(q,\lambda)\tau+ \tau\,m(q;\bu,\bz)\\
	& \quad -k(q;\bv,\bz) + \lambda\,m(q;\bv,\bz) + 2\mu m(q;\bu,\bv),
\end{align*}
where \(\tau\) is the test function of the eigenvalue and \(\bv\) is the test function of the eigenfunction.
using linearity of $k,m$ in the last two arguments and symmetry of $m$.
Setting $\mu=0$ and considering variations in $\bv$, we see that $\bz$ solves the adjoint eigenvalue problem
\[
k(q;\bv,\bz) =  \lambda\,m(q;\bv,\bz)\qquad\forall \bv\in \bU.
\]
Variations in $\tau$ yield a normalization of the adjoint eigenfunction
\[
m(q;\bu,\bz) = - J_{\lambda}'(q,\lambda).
\]
We obtain the derivative of the reduced cost functional with the thus computed values for $(\lambda,\bu,\mu,\bz)$
\begin{equation*}
	j'(q)p = {\mathcal{L}}_q' (q,({\lambda},u),(\mu, z))p \qquad \forall p\in \widetilde{Q}.
\end{equation*}
Assuming $\widetilde{Q} \subset Q$ to be dense and $j'(q)$ to be extendable to a functional
on $Q$, and not just on $\widetilde{Q}$, we can compute the $Q$-gradient $\nabla_Q j(q)$
by inverting the Riesz isomorphism, i.e., solving
\begin{equation*}
	(\nabla_Q j(q),p)_Q = j'(q)p\qquad\forall p\in Q.
\end{equation*}


\subsection{Derivative Formulas for the Domain Transformation}
\label{subsec:adjointCalc_MaxwellProblem}
In the previous Section~\ref{subsec:reducedderivative},
we derived an adjoint representation for the reduced gradient of a generalized eigenvalue optimization problem.
For the shape optimization problem of Maxwell's eigenvalue problem \eqref{EVP}, we consider the following cost functional
\begin{equation*}
	J(q, \lambda) = \frac{1}{2}\vert \lambda- \lambda_{*}\vert^2 + \frac{\alpha}{2}\left(\Vert q \Vert^2  + \Vert \nabla q  \Vert^2 \right)
	-\beta \int_{\widehat{\Omega}}\ln (\mathrm{J}_{q}-\varepsilon)\,\mathrm{d}\widehat{x},
	\label{cost}
\end{equation*}
where \(\lambda_{*} \in \mathbb{R}\) is a target eigenvalue and \(\alpha, \beta \in \mathbb{R}\)
are given parameters adding some regularization terms to the objective functional.
We apply a quadratic penalty regularization using the \(H^1\)-norm to guarantee the existence of \(q\) and \(\nabla q\). In addition, the deformation gradient \(\mathrm{DF}_q\), defined in~\eqref{deformationgradient}, needs to be invertible and its determinant to be non-negative. To ensure this, we apply a barrier term  enforcing \(\mathrm{J}_q = \det(\mathrm{DF}_q)\geq \varepsilon\) for some \(\varepsilon>0\).

Here, the state and adjoint needed simply coincide with the constraining eigenvalue problem for $(\lambda,\bu)$
and the corresponding adjoint eigenvalue problem for $\bz$ with a particular normalization, i.e.,
in the setting of Section~\ref{subsec:shapeopt}, the state variables are 
\begin{equation*}
	(\lambda,\bu) = (\lambda, (u,\psi)) \in \mathbb{R}\times (H_0(\widehat{\operatorname{curl}}; \widehat{\Omega})\times H^1_0(\widehat{\Omega})).  
\end{equation*}
solving
\begin{equation}\label{eq:state}
	\begin{aligned}
		a(q;u,v) + b(q;v,\psi) & = \lambda \,m(q;u,v) & \forall v &\in H_0(\widehat{\operatorname{curl}}; \widehat{\Omega}),\\
		b(q;u,\varphi) \qquad \qquad\quad &= 0 & \forall \varphi & \in H^1_0(\widehat{\Omega}),\\
		m(q;u,u) &= 1,
	\end{aligned}
\end{equation}
while the adjoint is 
\begin{equation*}
	\bz = (z,\chi)) \in H_0(\widehat{\operatorname{curl}}; \widehat{\Omega})\times H^1_0(\widehat{\Omega})
\end{equation*}
solving
\begin{equation}\label{eq:adjoint}
	\begin{aligned}
		a(q;v,z) + b(q;v,\chi) & = \lambda \,m(q;v,z) & \forall v &\in H_0(\widehat{\operatorname{curl}}; \widehat{\Omega}),\\
		b(q;z,\varphi) \qquad \qquad\quad &= 0 & \forall \varphi & \in H^1_0(\widehat{\Omega}),\\
		2 m(q;u,z) &= \lambda-\lambda_*,
	\end{aligned}
\end{equation}
where we had defined 
\begin{align*}
	a(q; u, v) &= (\curl  \mathcal{G}_q(u), \curl  \mathcal{G}_q(v))_{\Omega_q},\\
	b(q;u,\varphi) &= (\mathcal{G}_q(u),\nabla \mathcal{H}_q(\varphi))_{\Omega_q},\\
	m(q, u,v) &= (\mathcal{G}_q(u),\mathcal{G}_q(v))_{\Omega_q},
\end{align*}
where $\mathcal{H}_q,\mathcal{G}_q$ are the isomorphisms between reference and physical function spaces as defined in Section~\ref{subsec::mapping}.

More tedious is the calculation of the derivatives with respect to the domain transformation.
In view of the integral transformations at the end of Section~\ref{subsec:shapeopt} it is clear
that the, directional, derivatives of the deformation gradient $\mathrm{DF}_q$, its inverse $\mathrm{DF}_q^{-T}$,
and its determinant $\mathrm{J}_q= \det(\mathrm{J}_q)$ are needed, i.e., for $q \in Q^{\mathrm{ad}}$ and direction $p$, one needs
\begin{align*}
	\mathrm{DF}_q'd = \nabla d,\quad
	(\mathrm{F}_q^{-T}) d, \quad \mathrm{J}_q'd.
\end{align*}
As previously shown, the derivative of the reduced cost functional is given as
\begin{equation}\label{eq:gradient}
	\begin{aligned}
		j'(q) &= {\mathcal{L}}_q' (q,({\lambda},(u,\psi)),(0, (z,\chi)))p\\
		&= J'_q(q) p  - a'_q(q;u,z)p - b'_q(q;z,\psi)p - b'_q(q;u,\chi)p
		+ \lambda m'_q(q;u,z)p
	\end{aligned}
\end{equation}
already inserting $\mu = 0$.
The first term is 
\begin{align*}
	J'_q(q,\lambda) p  &= \alpha\left((q,p) + (\nabla q, \nabla p)\right) - \beta \int_{\widehat{\Omega}} \frac{1}{\mathrm{J}_q}(\mathrm{J}_q)'p\, \mathrm{d}\widehat{x},\\
	\intertext{	For the remaining terms, the derivatives follow from the respective formulas in Section~\ref{subsec:shapeopt}, i.e., for $d=2$}
	a'_q(q;u,z)p &\overset{d=2}{=} \int_{\widehat{\Omega}} \frac{-1}{\mathrm{J}_q^2}\mathrm{J}_q'p \left( \curl u,\curl z\right)\rm{d} \widehat{x},
	\intertext{	while for $d=3$}
	a'_q(q;u,z)p &\overset{d=3}{=} \int_{\widehat{\Omega}}\frac{-1}{\mathrm{J}_q^2}\mathrm{J}_q'p \,\left(\mathrm{DF}_{q}\,\curl u,\mathrm{DF}_{q}\,\curl z\right)\rm{d} \widehat{x}\\
	&+\int_{\widehat{\Omega}}\frac{1}{\mathrm{J}_q} \,\left(\mathrm{DF}_{q}'p\,\curl u,\mathrm{DF}_{q}\,\curl z\right)\rm{d} \widehat{x}\\
	&+\int_{\widehat{\Omega}}\frac{1}{\mathrm{J}_q} \,\left(\mathrm{DF}_{q}\,\curl u,\mathrm{DF}_{q}'p\,\curl z\right)\rm{d} \widehat{x}.
	\intertext{The remaining terms are independent of the dimension}
	b'_q(q;z,\psi) &= \int_{\widehat{\Omega}}\mathrm{J}_q'p \left(\mathrm{DF}_{q}^{-T}\,z, \mathrm{DF}_{q}^{-T}\,\nabla \psi\right)\rm{d} \widehat{x}\\
	&+\int_{\widehat{\Omega}}\mathrm{J}_q \left( (\mathrm{DF}_{q}^{-T})'p\,z, \mathrm{DF}_{q}^{-T}\,\nabla \psi\right)\rm{d} \widehat{x}\\
	&+\int_{\widehat{\Omega}}\mathrm{J}_q \left(\mathrm{DF}_{q}^{-T}\,z, (\mathrm{DF}_{q}^{-T})p\,\nabla \psi\right)\rm{d} \widehat{x},\\
	m'_q(q,u,z)p
	&=\int_{\widehat{\Omega}} \mathrm{J}_q'p\left(\mathrm{DF}_{q}^{-T}\,u,\mathrm{DF}_{q}^{-T}\,z\right)\rm{d} \widehat{x}\\
	&+\int_{\widehat{\Omega}} \mathrm{J}_q\left( (\mathrm{DF}_{q}^{-T})'p\,u,\mathrm{DF}_{q}^{-T}\,z\right)\rm{d} \widehat{x}\\
	&+\int_{\widehat{\Omega}} \mathrm{J}_q\left(\mathrm{DF}_{q}^{-T}\,u,(\mathrm{DF}_{q}^{-T})'p\,z\right)\rm{d} \widehat{x}.
\end{align*}

\section{A Damped Inverse BFGS Method} \label{sec:DampedInverseBFGS}
In order to solve the reduced optimization problem~\eqref{reducedEigenvalueProblem},
we consider a damped inverse Broyden-Fletcher-Goldfarb-Shanno (BFGS) method.
The damping was introduced for an update for the Hessian approximation by~\cite{Powell:1978},
here we show that an analogous approach can be utilized for the approximation of the inverse
Hessian. 
The resulting damped inverse BFGS method is given in Algorithm~\ref{BFGSInverseAlgorithm},
where it is assumed that the problem is such that a $Q$-gradient of the reduced objective
can be computed. 

\begin{remark}
	Note, that to assert global convergence, a check of the angle between $d^k$
	and $\nabla_Q j(q^k)$ could be added in the algorithm. However, in our numerical tests
	convergence to near stationarity was achieved without such safeguards.
\end{remark}

\begin{remark}
	In the implementation of Algorithm~\ref{BFGSInverseAlgorithm} the operator $B_k$ is never
	stored but the action $B_k \nabla_Q j(q^k)$ is evaluated using the recursive definition
	of $B_k$ given by the update formula. To this end, the functions $\widetilde{d}, B_ky^k \in Q$
	and the \mbox{scalars $(\widetilde{d}^k,y^k)_Q$} and $(\widetilde{d}^k-B_ky^k,y^k)$ are stored.
	To avoid unwanted increase in memory usage, the implementation allows to specify the
	number $m$ of previous iterates to be kept. The recursive definition then uses the definition
	$B_{k-m} = B_0$. To ensure compactness of \(B_0-\nabla^2j(q)\), we choose by default  $B_0 = 1/\alpha I$, where $I$ denotes
	the identity on $Q$. Hence, the BFGS method guarantees fast convergence, see \cite[Theorem~5.2]{Griewank:1987}, \cite[Theorem~2.5]{KelleySachs:1991}.
	
\end{remark}

In the following lemma, we show that the damping step ensures that positive definiteness of $B_k$ is sufficient to assert positive definiteness of \(B_{k+1}\).
\begin{lemma}[]
	Let  \(\widetilde{d}^k,y^k\) be given and assume that \(B_0\in \mathcal{L}(Q,Q)\) is a
	symmetric, positive definite operator. Then the damping step in
	Algorithm~\ref{BFGSInverseAlgorithm} ensures that \(\langle y^k, \widetilde{d}^k \rangle > 0\). By induction \(B_{k}\), and thus \(B_{k+1}\), remains positive definite.
\end{lemma}
\begin{proof}
	First, we verify the claim, that \(\langle y^k, \widetilde{d}^k \rangle > 0\) is sufficient
	to assert positive definiteness of $B_{k+1}$. The proof of this is almost analog to the
	finite dimensional case, see, e.g.,~\cite{UlbrichUlbrich:2012}.
	To this end, let $p\in Q\setminus \{0\}$, then
	the update formula asserts
	\begin{align*}
		(p,B_{k+1}p)_Q =&\; (p,B_kp)_Q + \frac{(p,\widetilde{d}^k - B_ky^k)_Q(\widetilde{d}^k,p)_Q +(p,\widetilde{d}^k)_Q(\widetilde{d}^k-B_ky^k,p )_Q}{(\widetilde{d}^k, y^k)_Q} \\
		&-  (p,\widetilde{d}^k)_Q\frac{(\widetilde{d}^k-B_ky^k, y^k)_Q}{(\widetilde{d}^k, y^k)_Q^2}(\widetilde{d}^k,p)_Q.\\
		\intertext{By expanding the second and third summand, it equals to }
		=&\; (p,B_kp)_Q + 2\frac{(p,\widetilde{d}^k)_Q^2 -(p,B_ky^k)_Q(p,\widetilde{d}^k)_Q}{(\widetilde{d}^k, y^k)_Q} \\
		&-  (p,\widetilde{d}^k)_Q^2\frac{(\widetilde{d}^k, y^k)_Q-(B_ky^k, y^k)_Q}{(\widetilde{d}^k, y^k)_Q^2}\\
		\intertext{Again, expanding and re-sort terms provides }
		=&\; (p,B_kp)_Q + \frac{(p,\widetilde{d}^k)_Q^2 }{(\widetilde{d}^k, y^k)_Q}
		- 2\frac{ (p,B_ky^k)_Q(p,\widetilde{d}^k)_Q}{(\widetilde{d}^k, y^k)_Q} 
		+  (p,\widetilde{d}^k)_Q^2\frac{(B_ky^k, y^k)_Q}{(\widetilde{d}^k, y^k)_Q^2}.\\
		\intertext{By adding \(\pm \frac{(B_ky^k,\widetilde{d}^k)_Q^2}{(y^k,B_ky^k)_Q}\), we obtain }
		=&\; (p,B_kp)_Q + \frac{(p,\widetilde{d}^k)_Q^2 }{(\widetilde{d}^k, y^k)_Q}
		- \frac{(B_ky^k,\widetilde{d}^k)_Q^2}{(y^k,B_ky^k)_Q}\\
		&+\frac{(B_ky^k,\widetilde{d}^k)_Q^2}{(y^k,B_ky^k)_Q}
		- 2\frac{ (p,B_ky^k)_Q(p,\widetilde{d}^k)_Q}{(\widetilde{d}^k, y^k)_Q} 
		+  (p,\widetilde{d}^k)_Q^2\frac{(B_ky^k, y^k)_Q}{(\widetilde{d}^k, y^k)_Q^2},\\
		\intertext{where we consider by binomial formula }
		=&\; (p,B_kp)_Q + \frac{(p,\widetilde{d}^k)_Q^2 }{(\widetilde{d}^k, y^k)_Q}
		- \frac{(B_ky^k,\widetilde{d}^k)_Q^2}{(y^k,B_ky^k)_Q}\\
		&+(B_ky^k, y^k)_Q\left[
		\frac{(B_ky^k,\widetilde{d}^k)_Q}{(B_ky^k,y^k)_Q}
		-  \frac{(p,\widetilde{d}^k)_Q}{(\widetilde{d}^k, y^k)_Q}.
		\right]^2
	\end{align*}
	Under the assumption that \(B_k\) is positive definite, only the third summand can be negative, using a square
	root $B_k^{1/2}$ of the positive definite operator $B_k$, see, e.g.,~\cite{Kubrusly2012}, and Cauchy-Schwarz inequality one obtains
	\begin{align*}
		(p,B_{k+1}p)_Q &\ge(p,B_kp)_Q + \frac{(p,\widetilde{d}^k)_Q^2 }{(\widetilde{d}^k, y^k)_Q}
		- \frac{(B_ky^k,\widetilde{d}^k)_Q^2}{(y^k,B_ky^k)_Q}\\
		&\ge \|B_k^{1/2}p\|_Q + \frac{(p,\widetilde{d}^k)_Q^2 }{(\widetilde{d}^k, y^k)_Q}
		- \frac{\|B_k^{1/2}y^k\|_Q^2\|B_k^{1/2}\widetilde{d}^k\|_Q^2}{\|B_k^{1/2}y^k\|_Q^2}
		\ge \frac{(p,\widetilde{d}^k)_Q^2 }{(\widetilde{d}^k, y^k)_Q} \ge 0.
	\end{align*}
	Now, either, $B_k^{1/2}y^k$ and $B_k^{1/2}\widetilde{d}^k$ are linear independent, then the
	second inequality coming from Cauchy-Schwarz inequality is strict, or there is some $t \ne 0$
	such that
	\[
	B_k^{1/2}p = tB_k^{1/2}\widetilde{d}^k
	\]
	and then
	\[
	\frac{(p,\widetilde{d}^k)_Q^2 }{(\widetilde{d}^k, y^k)_Q} =\frac{(B_k^{-1/2}B_k^{1/2}p,\widetilde{d}^k)_Q^2 }{(\widetilde{d}^k, y^k)_Q} = t^2 \frac{(B_k^{-1/2}B_k^{1/2}\widetilde{d}^k,\widetilde{d}^k)_Q^2 }{(\widetilde{d}^k, y^k)_Q}= t^2 \frac{(\widetilde{d}^k,\widetilde{d}^k)_Q^2 }{(\widetilde{d}^k, y^k)_Q} >0
	\]
	implying positive definiteness of $B_{k+1}$.
	
	Now, it remains to show that $(y_k,\widetilde{d}^k)_Q > 0$.
	In the case \(\theta = 1 \), clearly, with \(\xi \in (0,1)\),
	\begin{equation*}
		(y^k, \widetilde{d}^k)_Q \geq \xi \, (y^k, B_ky^k)_Q > 0.	
	\end{equation*}  
	Otherwise, if \(\theta \neq 1\), it is 
	\begin{align*}
		(y^k, \widetilde{d}^k)_Q =&\; (y^k, \theta_k \widetilde{d}^k + (1-\theta_k) B_ky^k)_Q\\
		=&\;(1-\xi)\frac{( y^k, B_ky^k )_Q}{( y^k, B_ky^k )_Q - ( y^k, \widetilde{d}^k)_Q}( y^k, \widetilde{d}^k )_Q \\
		&+ \left(1- (1-\xi)\frac{( y^k, B_ky^k )_Q}{( y^k, B_ky^k )_Q - ( y^k, d^k)_Q}\right )( y^k, B_k y^k )_Q \\
		=&\; (1-\xi)\frac{( y^k, B_ky^k )_Q}{( y^k, B_ky^k )_Q - ( y^k, \widetilde{d}^k)_Q}( y^k, \widetilde{d}^k )_Q \\
		&+  \frac{( y^k, B_ky^k )_Q - ( y^k, \widetilde{d}^k)_Q - (1-\xi) ( y^k, B_ky^k )_Q}{( y^k, B_ky^k )_Q - ( y^k, \widetilde{d}^k)_Q}( y^k, B_k y^k )_Q \\
		=&\;\frac{( y^k, B_ky^k )_Q \cdot \left((1-\xi) ( y^k, \widetilde{d}^k )_Q - ( y^k,\widetilde{d}^k)_Q + \xi( y^k, B_ky^k )_Q \right) }{( y^k, B_ky^k )_Q - ( y^k, \widetilde{d}^k)_Q}\\
		=&\; \frac{( y^k, B_ky^k)_Q - \xi( y^k,\widetilde{d}^k )_Q + \xi( y^k , B_ky^k)_Q}
		{( y^k, B_ky^k )_Q - ( y^k, \widetilde{d}^k)_Q} \\
		=&\; \xi( y^k, B_ky^k )_Q > 0.
	\end{align*}
	Similar to the finite dimensional damping, see \cite{UlbrichUlbrich:2012}, we choose  \(\xi = 0.2\).
	Thus the assertion is shown.
\end{proof}

\begin{algorithm}
	\caption{Damped inverse BFGS method for an eigenvalue shape optimization problem}	
	\label{BFGSInverseAlgorithm}
	\begin{algorithmic}[1] 
		\REQUIRE Let \(q^0\in Q\) be an initial guess for the control
		and the initial inverse Hessian approximation be \(B_0 \in \mathcal{L}(Q,Q)\) (symmetric, positive definite). We choose the parameter TOL > 0, \(k_{\max} \in  \mathbb{N}\),
		and \(\gamma \in (0, 0.5)\).
		\WHILE{\(\Vert \nabla_Q j(q^k) \Vert_Q > \text{TOL} \)}
		\STATE Solve the  state eigenvalue problem~\eqref{eq:state} to obtain \(\lambda, u \in \mathbb{R}\times U\).
		\STATE Solve the adjoint eigenvalue problem~\eqref{eq:adjoint} to obtain \(\overline{\lambda}, z \in \mathbb{R}\times Z\).
		\STATE  Compute the gradient \(\nabla_Q j(q^k)\), using~\eqref{eq:gradient}.
		\STATE  Compute the search direction
		\begin{equation*}
			d^k = -B_k\nabla_Q j(q^k).
		\end{equation*}
		\STATE Set \(q_{k+1} = q^k + t_kd^k\), where \(t_k\) is computed by backtracking line search, i.e., $t_k = \max(1,\rho,\rho^2,\ldots)$ for some $\rho\in (0,1)$ such that 
		\begin{equation*}
			j(q^k - t_kd^k) \leq j(q^k) + \gamma t_k (\nabla_Qj(q^k),d^k)_Q
		\end{equation*}
		holds.
		\STATE Set \(\widetilde{d}^k=t_kd^k = q^{k+1} - q^k\)  and \(y^k = \nabla_Q j(q^{k+1}) - \nabla_Q j(q^k)\).
		\STATE \textit{Damping step:}
		\IF{\( (y^k, \widetilde{d}^k)_Q >= \xi (y^k, B_k y^k)_Q\)}
		\STATE \(\theta_k = 1\)
		\ELSE
		\STATE	\(\theta_k =  (1-\xi)\frac{(y^k, B_k y^k)_Q}{(y^k, B_k y^k)_Q - (\widetilde{d}^k, y^k)_Q}\)
		\ENDIF
		\STATE \(\widetilde{d}^k = \theta_k \widetilde{d}^k + (1-\theta_k) B_ky^k \)
		\STATE  Update $B_{k+1}\in \mathcal{L}(Q,Q)$ such that for any $p \in Q$ it holds
		\begin{align*}
			B_{k+1}p &= B_kp + \frac{(\widetilde{d}^k - B_ky^k)(\widetilde{d}^k,p)_Q +\widetilde{d}^k(\widetilde{d}^k-B_ky^k,p )_Q}{(\widetilde{d}^k, y^k)_Q} \\
			&\quad-  \widetilde{d}^k\frac{(\widetilde{d}^k-B_ky^k, y^k)_Q}{(\widetilde{d}^k, y^k)_Q^2}(\widetilde{d}^k,p)_Q
		\end{align*}
		\STATE Set \(k \leftarrow k+1\).
		\ENDWHILE
	\end{algorithmic}
	\end{algorithm}	
	
	
	\section{Numerical Examples on a 5-cell Cavity}
	\label{sec:examples}
	In this section, we show numerical results of the damped inverse BFGS algorithm from Algorithm~\ref{BFGSInverseAlgorithm} using adjoint calculus in order to solve the eigenvalue optimization problem \eqref{EVP}. 
	We consider a two-dimensional planar 5-cell cavity geometry inspired by a low \(\beta\) cavity design from \cite[Chapter 6.4]{NumMetForEst} on a model for the Superconducting-DArmstadt-LINear-ACcelerator (S-DALINAC)\footnote{\url{http://www.ikp.tu-darmstadt.de/sdalinac_ikp/}}.
	Moreover, we consider to optimize the first (smallest) eigenvalue \(\lambda_0\) to a target eigenvalue \(\lambda_{*}\). 
	We clarify here, that the chosen eigenvalue to optimize is not necessary the first one. The only restriction to the theory is that we choose a simple eigenvalue. 
	
	We discretize the function spaces with a mixed finite element method. We discretize the space of \(H^1_0 (\Omega)\) with Lagrange elements. To ensure the tangential continuity of the \(H_0 (\operatorname{curl}; \Omega)\), we discretize the function space with Nédélec elements. 
	For more details for finite elements and their properties, we refer to, e.g., ~\cite{monk2003finite,Boffi2010FiniteEA}. 
	For the numerical computation, we use the \texttt{C++}-based optimization library \texttt{DOpElib} \cite{DOpElib}, which bases on the finite element library \texttt{deal.II}~\cite{dealII92,dealII91}.
	Further, the eigenvalue problems within the optimization are solved by the Krylov Schur  eigenvalue solver with a shift-invert spectral transformation provided in the \texttt{SLEPc} library version 3.18.2 \cite{slepc2,slepc1} which bases on the \texttt{PETSc} library version 3.18.5 \cite{petsc-user-ref,petsc-web-page,petsc-efficient}. The convergence tolerance for the eigenvalue solver is set to \(10^{-5}\). 
	The eigenvalue solver applies a Richardson Krylov solver and the Cholesky preconditioner chosen with the default properties of the \texttt{SLEPc}- and \texttt{PETSc}-libraries for the first example, whereas in the second example, we set the relative convergence tolerance of the Richardson Krylov solver to \(10^{-5}\) and allow this method to automatically determine optimal scaling at each iteration to minimize the 2-norm of the preconditioned residual.
	
	In the following, we consider the starting design of a so-called low \(\beta\) cavity, where the cells get longer in order to take into account the increasing speed of the particle bunch. The geometry is described with the parameters shown in \Cref{tab:cavitygeometry} and the initial domain of the 5-cell cavity is shown in \Cref{fig:geometry}. 
	\begin{figure}[h!]
		\centering
		\includegraphics[width=0.4\linewidth]{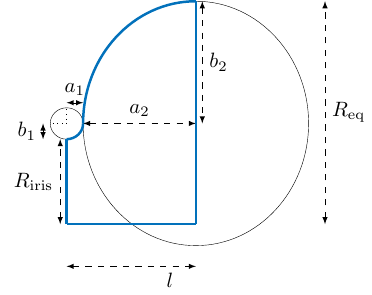}
		\caption[Geometry of a Cavity.]{Geometry of a Cavity.}
		\label{fig:geometry}
	\end{figure}
	\begin{table}[h!]
		\caption[Cavity design parameters for the different cells, see \cite{NumMetForEst}.]{Cavity design parameters for the different cells, see \cite{NumMetForEst}. All dimensions are given in \si{\milli\m}.}
		\label{tab:cavitygeometry}\centering
		\begin{tabular}[]{p{4.65cm}|p{1.225cm}|p{1.225cm}|p{1.225cm}|p{1.225cm}|p{1.225cm}}
			Cavity Shape Parameter & Cell 1 & Cell 2 & Cell 3 & Cell 4 & Cell 5\\ \hline
			Equator radius \(R_{\mathrm{eq}}\) 	& 43.44542 & 43.44542 & 43.44542 & 43.44542 & 43.44542\\
			Iris radius \(R_{\mathrm{iris}}\)& 16.5627 	&16.5627 	&16.5627 	& 16.5627	&16.5627 \\
			Horizontal half axis at iris \(a_1\) 	& 3.2 		& 3.7184 	& 4.32 		& 5.0208 	& 5.8336 \\
			Vertical half axis at iris \(b_1\) 	& 3.0592 	& 3.0592	& 3.0592	& 3.0592	& 3.0592\\
			Horizontal half axis at equator \(a_2\) 	& 21.98 	& 25.54076 & 29.673 	& 34.48662 & 40.06954\\
			Vertical half axis at equator \(b_2\) 	& 23.82352 & 23.82352 & 23.82352 & 23.82352 & 23.82352\\
			Length \(L \)		& 25.18 	& 29.25916 & 33.993 	& 39.50742 & 45.90314
		\end{tabular}
	\end{table}
	\begin{figure}[h!]\centering
		\includegraphics[width=\linewidth]{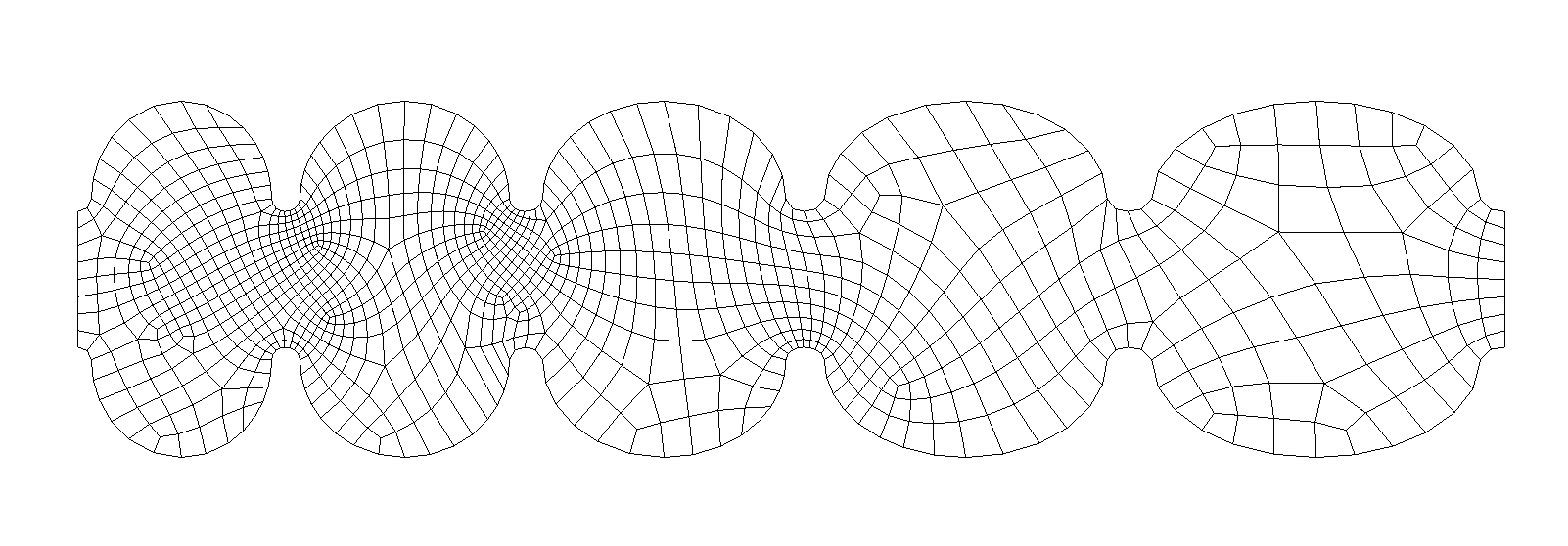}\\
		\caption[Mesh of the 5-cell Cavity]{Mesh of the 5-cell Cavity}
		\label{fig:mesh_cavity}
	\end{figure}
	The meshes of the geometries are created with GMSH~\cite{gmesh} and shown in \Cref{fig:mesh_cavity}.
	Further, we set the maximum number of iterations of the BFGS method \mbox{to 100,} the termination tolerance is set to \(\text{TOL } = 10^{-7}\). For the Armijo line search, we set the maximum number of iterations to 10 and \(\gamma = \rho = 0.1\).
	Furthermore, we set the target value \(\lambda_{*} = 6017\) and the regularization parameters \(\alpha = 100, \, \beta = 10^{-6}\) \mbox{and \(\varepsilon = 10^{-4}\).}
	The results in \Cref{tab:5cellcav} show that by a discretization with 22103 DoFs or higher, the BFGS method converges after 16 iterations. For a setting \mbox{with 5438 DoFs,} it takes 12. The values of \(\mathrm{J}_q\) are slightly smaller than 1, which means that we obtain a small shrinkage of the domain. The deformation of a 5-cell cavity after optimization is shown in \Cref{Fig:5cellcav}.
	
	\begin{table}[h!]
		\caption[Results of a freeform optimization with
                target value  \(\lambda_{*} = 6017\) using the BFGS
                method from \cref{BFGSInverseAlgorithm} with \(\alpha
                = 100\) and \(\beta= 10^{-6}\) on a 5-cell
                cavity.]{Results of a freeform optimization with
                  target value  \(\lambda_{*} = 6017\) using the BFGS
                  method from \cref{BFGSInverseAlgorithm} with
                  \(\alpha = 100\) and \(\beta= 10^{-6}\) on a 5-cell cavity. The initial first (numerical) eigenvalue is given by \(\lambda_{0}\). We show the eigenvalue \(\lambda_{\mathrm{it}}\) after termination of the method and the resulting value of the cost functional \(J\), the relative residual of the gradient of the cost functional as well as the minimum and maximum value of the determinant of the deformation gradient \(\mathrm{J}_q\).}
		\setlength{\tabcolsep}{2pt}
		\begin{tabular}{l|c|c|c|l||l|l|l|l|l|l}
			DoFs & ref. & Lagr. & Néd. & \(\lambda_{0}\)  & it. & \(\lambda_{\mathrm{it}.}\) & \( J \)& \(\textrm{r}_{\mathrm{rel}}\) &\(\mathrm{J}_{q,\min}\)  & \(\mathrm{J}_{q,\max}\) \\ \hline \hline
			5438  & 0 & 1 & 0 & 6018.47 & 12 & 6017 & 1.728e-10 & 7.846e-07 & 0.99998 & 0.99999 \\
			22103 & 0 & 2 & 1 & 6022.89 & 16 & 6017 & 2.691e-09 & 9.067e-07 & 0.99987 & 0.99999\\			\hline
			22103 & 1 & 1 & 0 & 6021.91 & 16 & 6017 & 1.878e-09 & 7.243e-07 & 0.99993 & 0.99998 \\ \hline
			86723 & 2 & 1 & 0 & 6024.65 & 16  & 6017 & 4.532e-09 &  9.520e-07 & 0.99991 & 0.99998 \\ 
		\end{tabular}
		\label{tab:5cellcav}
	\end{table}
	
	\begin{figure}[htp]
		\centering
		\includegraphics[width=0.925\linewidth]{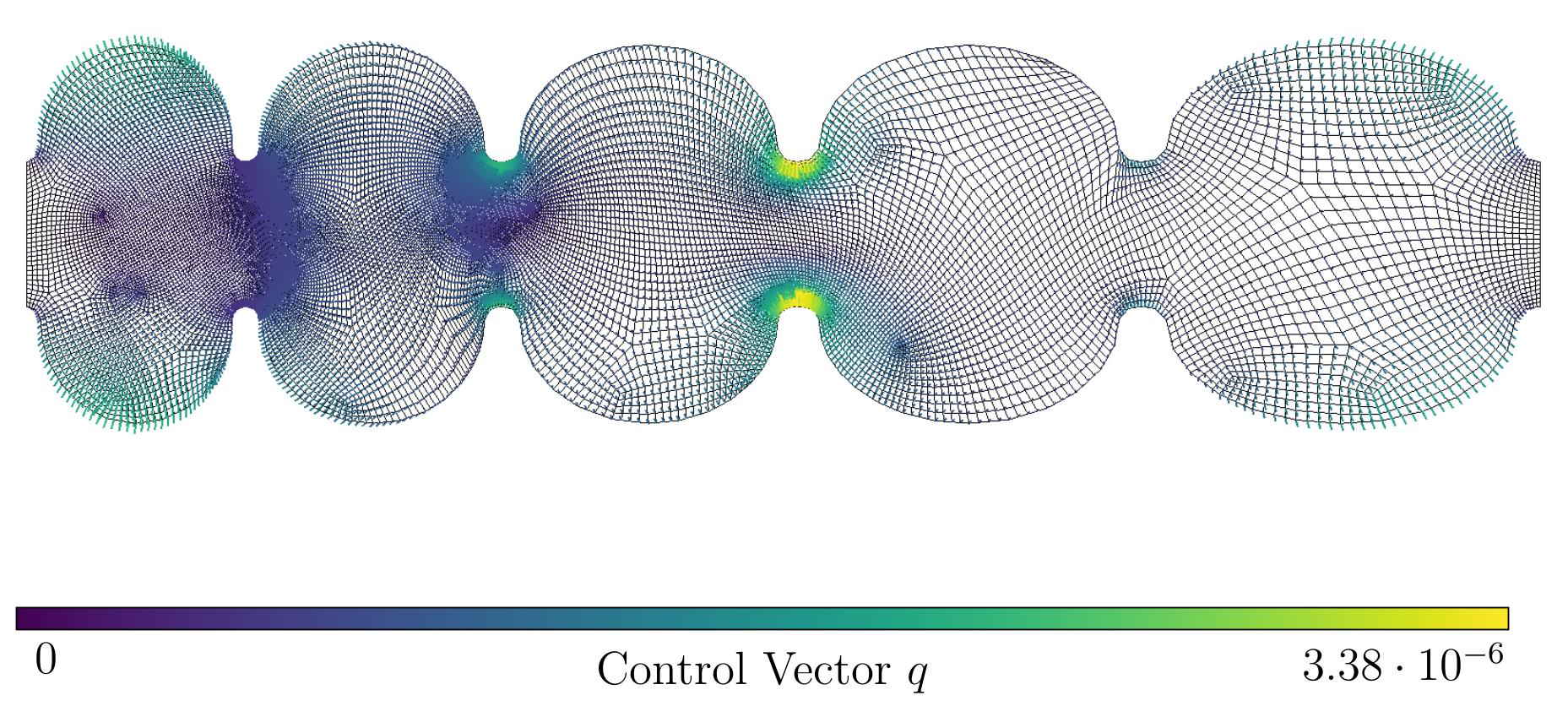}
	\caption[Deformation of a 5-cell cavity after optimization of the first eigenvalue to target value \(\lambda_{*} = 6017\) using the BFGS method from \cref{BFGSInverseAlgorithm} with regularization parameters \(\alpha = 100\) and \(\beta = 10^{-6}\).]{Deformation of a 5-cell cavity after optimization of the first eigenvalue to target value \(\lambda_{*} = 6017\) using the BFGS method from \cref{BFGSInverseAlgorithm} with regularization parameters \(\alpha = 100\) and \(\beta = 10^{-6}\). 
		The chosen refinement level is 2, the number of DoFs is 86723. Lagrange elements of order 1 and lowest order Nédélec elements are used.}
	\label{Fig:5cellcav}
\end{figure}

\section{Conclusion}
In this paper, we considered an eigenvalue optimization problem with respect to shape-variations in electromagnetic systems motivated by particle accelerator cavities, such as superconducting TESLA or low \(\beta\) cavities, see \cite{PhysRevSTAB,NumMetForEst, UncertaintyQuantiNiklas}.
We formulated an optimization problem with the mixed formulation by Kikuchi and a domain mapping, where we distinguished between the function spaces of \(H_0(\text{curl})\) and \(H^1_0\). 
We discussed the approach of adjoint calculus for general eigenvalue optimization problems and  we applied adjoint calculus to the concrete problem.
In order to solve this optimization problem, we discussed a damped inverse BFGS method for infinite dimensional problems.
We proved the preservation of the positive definiteness property of the updating operator which ensures that the curvature condition is fulfilled.
Finally, we showed the functionality of this method on a numerical example of an optimization of a cavity domain.

	\bibliographystyle{abbrvnat}
        \bibliography{bibdatabase.bib}

\end{document}